\documentclass[11pt]{amsart}
\usepackage{oldgerm}
\usepackage{amssymb} 
\usepackage{mathrsfs} 
\usepackage{amsmath}
\usepackage{latexsym}
\usepackage[all]{xy}
\usepackage[dvips]{graphics}
\usepackage{amsthm}
\usepackage{enumerate}

\allowdisplaybreaks

\theoremstyle{plain} 
\newtheorem{thm}{Theorem}[section] 
\newtheorem*{thm*}{Theorem}
\newtheorem{prop}[thm]{Proposition} 
\newtheorem{lem}[thm]{Lemma} 
\newtheorem{cor}[thm]{Corollary} 
\newtheorem{conj}[thm]{Conjecture}
\theoremstyle{definition} 
\newtheorem{defn}[thm]{Definition} 
\newtheorem{rem}[thm]{Remark} 
\newtheorem*{acknow}{Acknowledgements} 
\theoremstyle{remark} 
\newtheorem*{notation}{Notation} 
\title[the trace of modular forms and its application
]
{the trace of modular forms\\
and its application to number theory
}

\author
{Norifumi Ojiro}
\address{Department of Mathematics, 
Graduate School of Science, 
Hiroshima University \\
1-3-1, Kagamiyama, Higashi-Hiroshima, Hiroshima, 739-8526, Japan}
\email{d153696@hiroshima-u.ac.jp}

\date{5 April 2016}


\begin{document}
\maketitle

\begin{abstract}
We provide a generalization of an algebraic linear combination for the trace of certain elliptic modular forms, and through specializing the expression at a suitable pair consisting of an elliptic curve over algebraic number fields and its a certain cyclic subgroup with finite order, show a formula
between distinct algebraic number fields, the one related to modular forms and the other related to elliptic curves. 
\end{abstract}

\section{Introduction}
It is classically known that elliptic modular functions play interesting roles in number theory such as some important results for algebraic number fields are explicitly described by using these special values. Similarly for elliptic modular forms, its an application to number theory is mentioned as below.

First let us explain necessary terms in brief. Every element $\alpha$ of a subgroup ${\rm GL}_2^+(\mathbf R)$ of 
${\rm GL}_2(\mathbf R)$ consisting of all matrices with positive determinant, acts on a meromorphic function $h$ on the complex upper half plane $\mathfrak H$ as follows:
\begin{eqnarray*}
h(z)\mid_m\alpha&=&\det(\alpha)^{m/2}h(\alpha\,z)j(\alpha,z)^m,
\end{eqnarray*}
where $m$ is an integer, by $\alpha\,z$ we mean the linear fractional transformation on $\mathfrak H$, and put the factor of automorphy as $j(\alpha,z)=(cz+d)^{-1}$ for the second row $(c,d)$ of $\alpha$. 

Let $k$, $l$ and $N$ be positive integers and always assume $N$ so throughout the present paper. We deal with a discrete group defined by
$$
\varGamma_0(N)=\left\{\left.\begin{pmatrix}a&b\\c&d\end{pmatrix}\in{\rm SL}_2(\mathbf Z)\right|c\equiv0\pmod N\right\}.
$$
Then $S_k(\varGamma_0(N))$ and $G_k(\varGamma_0(N))$ denote the $\mathbf C$-linear spaces consisting of all cusp forms and holomorphic modular forms of weight $k$ for $\varGamma_0(N)$, respectively, and $S_k(\varGamma_0(N))$ is furnished with the Petersson inner product $\langle*,*\rangle$.

for any $f\in S_k(\varGamma_0(N))$ and positive divisor $M$ of $N$, we consider the trace of $f$ to level $M$ defined by
$$
{\rm Tr}^N_M(f)=\sum_{\gamma\in\varGamma_0(N)\backslash\varGamma_0(M)}f\mid_k\gamma\in S_k(\varGamma_0(M)),
$$
and for a positive integer $\lambda$ with $\lambda>2$, take the following Eisenstein series
$$
E_{\lambda,N}(z)=\sum_{\gamma\in\varGamma_\infty\backslash\varGamma_0(N)}j(\gamma,z)^\lambda\in G_{\lambda}(\varGamma_0(N)),
$$
where $\varGamma_{\infty}$ is a subgroup of $\varGamma_0(N)$ consisting of all elements fixing the point at infinity, namely
$$
\varGamma_{\infty}=\left\{\pm\left.\begin{pmatrix}1&n\\0&1\end{pmatrix}\right| n\in\mathbf Z\right\}.
$$ 
for any $f\in S_k(\varGamma_0(N))$ and $g\in G_l(\varGamma_0(N))$, the associated zeta function is given by the series
$$
D(s, f, g)=\sum_{n=1}^\infty c_n(f)c_n(g)n^{-s}\ \ {\rm for\ any}\ s\in\mathbf C,
$$
where $c_n(f)$ and $c_n(g)$ are the $n$-th Fourier coefficient of $f$ and $g$, respectively. 
Then the starting point of the paper is the following.
\begin{thm}[cf. {\cite[Theorem]{DHM}}]\label{dhm}
Let $k$, $l$, $\lambda$ and $\mu$ be positive integers with $k=l+\lambda$ and $\lambda>2$. for any $g\in S_l(\varGamma_0(N))$, we have
\begin{equation}
\mathrm{Tr}^N_1((gE_{\lambda,N})^\mu)=c_1
\sum_{i=1}^{r}\frac{D(k\mu-1,f_i,g^\mu E_{\lambda,N}^{\mu-1})}{\pi^{k\mu}\langle f_i,f_i\rangle}f_i, 
\label{e0.0}
\end{equation}
where $c_1=3\cdot4^{-(k\mu-1)}(k\mu-2)!$ and $\{f_1,\cdots,f_r\}$ is a unique basis of $S_{k\mu}({\rm SL}_2(\mathbf Z))$ consisting of primitive forms.
\end{thm}
When we take the above $g$ as all the Fourier coefficients belong to $\mathbf Q$, by {\cite[Theorem 3]{S2}} see that each coefficient of the right-hand side of the equation \eqref{e0.0} belongs to the Hecke field $\mathbf Q_{f_i}$ that is the algebraic number field generated by adding all coefficients of $f_i$ to $\mathbf Q$. 
That is, the equation \eqref{e0.0} is a linear combination with algebraic coefficients for the trace of modular forms to level $1$.
 
On the other hand, the transformation polynomial for $gE_{\lambda,N}\in S_{k}(\varGamma_0(N))$ to level $1$ is defined by
$$
\Phi^N_1(X;gE_{\lambda,N})=\hspace{-2mm}\prod_{\gamma\in\varGamma_0(N)\backslash{\rm SL}_2(\mathbf Z)}\hspace{-3mm}\left(X-gE_{\lambda,N}\mid_k\gamma\right)=\sum_{i=0}^{\mu_N}(-1)^{i}s_i(gE_{\lambda,N})X^{\mu_N-i}.
$$
Here $\mu_N=[{\rm SL}_2(\mathbf Z):\varGamma_0(N)]$ is the index of $\varGamma_0(N)$ in ${\rm SL}_2(\mathbf Z)$, and the coefficient $s_i(gE_{\lambda,N})$ is the elementary symmetric polynomial of degree $i$ with respect to $\left\{gE_{\lambda,N}\mid_k\gamma\ ; \gamma\in\varGamma_0(N)\backslash{\rm SL}_2(\mathbf Z)\right\}$; it belongs to $S_{ki}({\rm SL}_2(\mathbf Z))$.

Let ${\rm E}_{\mathbf Q}$ be an elliptic curve over $\mathbf Q$ given by a Weierstrass equation and $\mathbf C/L$ the corresponding complex torus with $L=\mathbf Z\omega_1+\mathbf Z\omega_2$ and $\omega_1/\omega_2\in\mathfrak H$. Then we define the specialization of $f\in S_{k\mu}({\rm SL}_2(\mathbf Z))$ and $\Phi^N_1(X;gE_{\lambda,N})$ at ${\rm E}_{\mathbf Q}$ as follows:
\begin{eqnarray*}
f({\rm E}_{\mathbf Q})&=&(2\pi/\omega_2)^{k\mu}f(\omega_1/\omega_2),\\
\Phi^N_1(X;gE_{\lambda,N},{\rm E}_{\mathbf Q})&=&\sum_{i=0}^{\mu_N}(-1)^{i}(2\pi/\omega_2)^{ki}s_i(gE_{\lambda,N})(\omega_1/\omega_2)X^{\mu_N-i}.
\end{eqnarray*}
As it will see later, the specialized polynomial $\Phi^N_1(X;gE_{\lambda,N},{\rm E}_{\mathbf Q})$ belongs to $\mathbf  Q[X]$ and we may choose $g$, ${\rm E}_{\mathbf Q}$ as it is irreducible over $\mathbf Q$ (cf. Remarks \ref{rem-1} and \ref{rem-2}). Moreover, we assume that $S_{k\mu}({\rm SL}_2(\mathbf Z))$ is spaned by all ${\rm Aut}(\mathbf C)$-conjugates of an element $f$ of $\{f_1,\cdots,f_r\}$. Then it was suggested by \cite{DHM} that a nontrivial equality between distinct algebraic number fields is obtained by specializing the equation \eqref{e0.0} at ${\rm E}_{\mathbf Q}$:
\begin{eqnarray}
\lefteqn{
{\rm Tr}_{\mathbf Q_N/\mathbf Q}\left\{(2\pi/\omega_2)^k gE_{\lambda,N}(\omega_1/\omega_2)\right\}^\mu}\label{e0.1}\\
&=& c_1{\rm Tr}_{\mathbf Q_f/\mathbf Q}\left\{\frac{D(k\mu-1,f,g^{\mu}E_{\lambda,N}^{\mu-1})}{\pi^{k\mu}\langle f,f\rangle}
(2\pi/\omega_2)^{k\mu}f(\omega_1/\omega_2)
\right\},\nonumber
\end{eqnarray}
where $\mathbf Q_N$ is an algebraic number field associated with the $J$-invariant of $E_{\mathbf Q}$, namely
\[
\mathbf Q_N={\mathbf Q}(J(\omega_1/\omega_2), J(N\omega_1/\omega_2)).
\]
We may interpret the formula \eqref{e0.1} as an equation which bridges the gap between the fields $\mathbf Q_N$ related to elliptic curves and $\mathbf Q_f$ related to modular forms via the trace to $\mathbf Q$, where note that $\mathbf Q_N$ is independent of $\mu$, although $\mathbf Q_f$ depends on it. Meanwhile, by the equations \eqref{e0.0} and \eqref{e0.1}, we see that all coefficients of the transformation polynomial $\Phi^N_1(X;gE_{\lambda,N})$ and its specialization $\Phi^N_1(X;gE_{\lambda,N},{\rm E}_{\mathbf Q})$ are expressed by using the special values of the associated zeta functions.

Now, an aim of the paper is to show a generalization of the equation \eqref{e0.0} to the case of the trace of cusp forms with a Dirichlet character to higher levels (cf. Theorem \ref{gdhm}). This objective is achieved by using the constant of each space of cusp forms uniquely determined by applying Atkin-Lehner theory. Another aim is to generalize the formula \eqref{e0.1} as application of our result (cf. Corollary \ref{gdhm-c}), and to derive the formula \eqref{e0.1} as a special case (cf. Corollary \ref{dhm-c}). In more detail, for modular forms with a Dirichlet character and suitable algebraic Fourier coefficients, we obtain a near formula to the formula \eqref{e0.1} by specializing the generalized equation at a pair consisting of an elliptic curve over the algebraic number fields and its a certain cyclic subgroup with finite order.

\begin{notation}
We mean ${\mathbf i}$ as $\sqrt{-1}$, and denote by $\overline z$, $|z|$, ${\rm Re}(z)$ and ${\rm Im}(z)$ the complex conjugate, the absolute value, the real part and the imaginary part of a complex number $z$, respectively. For two positive integers $M$ and $N$, the relation $M\mid N$ means that $M$ divides $N$. For two number fields $F$ and $K$, we denote the composite field of $F$ and $K$ by $FK$.
\end{notation}

\section{Preliminaries}
Let $k$, $l$ be positive integers and $\chi$, $\psi$ Dirichlet characters modulo $N$. Then $S_k(\varGamma_0(N),\chi)$ and $G_k(\varGamma_0(N),\chi)$ denote the $\mathbf C$-linear spaces consisiting of all cusp forms and holomorphic modular forms with $\chi$ of weight $k$ for $\varGamma_0(N)$, respectively, that is, those every element $g(z)$ satisfies the following condition:
$$
g\mid_{k}\gamma=\chi(\gamma)g\ \ {\rm for\ any}\ \gamma=\begin{pmatrix}a&b\\c&d\end{pmatrix}\in\varGamma_0(N),\ {\rm where\ put}\ \chi(\gamma)=\chi(d).
$$
As basic rules, by $\chi\psi$ and $\overline{\chi}$ we denote $\chi\psi(\gamma)=\chi(\gamma)\psi(\gamma)$ and $\overline{\chi}(\gamma)=\overline{\chi(\gamma)}$ for any $\gamma\in\varGamma_0(N)$, then note that $\chi\overline{\chi}$ is trivial (i.e. $\chi\overline{\chi}(\gamma)=1$). 

For any $f(z)\in S_k(\varGamma_0(N),\chi)$ and $g(z)\in G_l(\varGamma_0(N),\psi)$, whose Fourier expansions written as follows:
\[
f(z) = \sum_{n = 1}^{\infty} c_n(f) \exp(2\pi {\mathbf i}nz), \ \ \ \ 
g(z) = \sum_{n = 0}^{\infty} c_n(g) \exp(2\pi {\mathbf i}nz). \ \ \ \ 
\]
The associated Rankin-Selberg zeta function is defined by
\[
D(s,f,g)=\sum_{n=1}^\infty c_n(f)c_n(g)n^{-s} \ \ {\rm for\ any}\ s \in {\mathbf C}.
\]
For any integer $\lambda$ and complex number $s$, we define the following series
$$
E_{\lambda,N}(z,s,\chi)=\sum_{\gamma\in\varGamma_{\infty}\backslash\varGamma_0(N)}\overline{\chi(\gamma)}j(\gamma,z)^\lambda|j(\gamma,z)|^{2s}\ \ {\rm for\ any}\ z\in\mathfrak H.
$$
If $\lambda$, $s$ satisfy $\lambda+{\rm Re}(2s)>2$, the series absolutely and uniformly converges on $\mathfrak H$. In particular, $E_{\lambda,N}(z,0,\chi)$ with $\lambda>2$ belongs to $G_\lambda (\varGamma_0(N),\chi)$ and is called the Eisensten series. We simply denote $E_{\lambda,N}(z,0,\chi)$ by $E_{\lambda,N,\chi}(z)$, furthermore $E_{\lambda,N}(z)$ whenever $\chi$ is trivial. 

We put $f^\rho(z)=\overline{f(-\overline z)}$. Then it is easy to see that $f^\rho$ belongs to $S_k(\varGamma_0(N),\overline\chi)$ and its Fourier coefficient is expressed by $\overline{c_n(f)}$ for any positive integer $n$. These notions are mutually connected on the following.
\begin{lem}\label{rs}
Let $k$, $l$ be positive integers. for any $f\in S_k(\varGamma_0(N),\chi)$ and $g\in S_l(\varGamma_0(N),\psi)$, the zeta function $D(s,f,g)$ is expressed by an integral:
\begin{eqnarray*}
\lefteqn{
(4\pi)^{-s}\Gamma(s)D(s,f,g)
}\\
&=&\int_{\varGamma_0(N)\backslash\mathfrak{H}}\overline{f^{\rho}(z)}g(z)E_{k-l,N}(z,s+1-k,\overline{\chi\psi})({\rm Im}(z))^{s+1}dv(z), 
\end{eqnarray*}
where $dv(z)$ is a ${\rm GL}_2^{+}(\mathbf R)$-invariant measure on $\mathfrak{H}$ defined by $dv(z)=y^{-2}dxdy$ for all $z=x+{\bf i}\,y\in\mathfrak H$ and $\Gamma$ is the gamma function.
\end{lem}
\begin{proof}
Since $\overline{f^\rho}$ and $g$ are holomorphic on $\mathfrak{H}$, these Fourier expansions are termwise integrable for $x$ from $0$ to $1$. Hence we have 
\begin{eqnarray*}
\int_{0}^1\overline{f^\rho(z)}g(z)dx\hspace{-2mm}&=&\hspace{-2mm}\sum_{n=1}^{\infty}\sum_{m=0}^\infty c_n(f)c_m(g)\exp(-2\pi(n+m)y)\int_0^1\exp(2\pi{\bf i}(-n+m)x)dx\\
&=&\hspace{-2mm}\sum_{n=1}^{\infty}c_n(f)c_n(g)\exp(-4\pi ny).
\end{eqnarray*}
For any $s\in\mathbf C$ and $t\in\mathbf R$ with $t>1$, the series $\sum_{n=1}^\infty c_n(f)c_n(g)y^{s-1}\exp(-4\pi ny)$ absolutely and uniformly converges on $t^{-1}\le y\le t$. Therefore we have
$$
\int_{t^{-1}}^ty^{s-1}\int_0^1\overline{f^\rho(z)}g(z)dxdy=\sum_{n=1}^\infty c_n(f) c_n(g)\int_{t^{-1}}^ty^{s-1}\exp(-4\pi ny)dy.
$$
After a simple change of variable, by taking $t$ to $\infty$, we have the following
\begin{eqnarray*}
\int_0^\infty y^{s-1}\int_0^1\overline{f^\rho(z)}g(z)dxdy=(4\pi)^{-s}\Gamma(s)D(s,f,g).
\end{eqnarray*}
Furthermore, by the disjoint partition $\varGamma_\infty\backslash \mathfrak{H}=\bigsqcup_{\gamma\in \varGamma_\infty\backslash\varGamma_0(N)}\gamma\cdot(\varGamma_0(N)\backslash \mathfrak{H})$, the left-side hand of the above equation is transformed as follows:
\begin{eqnarray*}
\lefteqn{
\int_0^\infty y^{s-1}\int_0^1\overline{f^\rho(z)}g(z)dxdy=\int_{\varGamma_\infty\backslash \mathfrak{H}}y^{s+1}\overline{f^\rho(z)}g(z)dv(z) 
}
\\
&=&\sum_{\gamma\in \varGamma_\infty\backslash\varGamma_0(N)}
\int_{\varGamma_0(N)\backslash\mathfrak{H}}\overline{f^\rho(\gamma z')}g(\gamma z')({\rm Im}(\gamma z'))^{s+1}d(\gamma z')\\
&=&\int_{\varGamma_0(N)\backslash\mathfrak{H}}\sum_{\gamma\in \varGamma_\infty\backslash\varGamma_0(N)}\overline{f^\rho(\gamma z')}g(\gamma z')
\overline{j(\gamma,z')}^{s+1}j(\gamma,z')^{s+1}({\rm Im}(z'))^{s+1}dv(z')\\
&=&\int_{\varGamma_0(N)\backslash\mathfrak{H}}\hspace{-5mm}\overline{f^\rho(z')}g(z')\hspace{-3mm}\sum_{\gamma\in\varGamma_\infty\backslash\varGamma_0(N)}
\hspace{-3mm}\chi(\gamma)\psi(\gamma)j(\gamma,z')^{k-l}|j(\gamma,z')|^{2(s+1-k)}({\rm Im}(z'))^{s+1}dv(z')\\
&=&\int_{\varGamma_0(N)\backslash\mathfrak{H}}\overline{f^\rho(z)}g(z)E_{k-l,N}(z,s+1-k,\overline{\chi\psi})({\rm Im}(z))^{s+1}dv(z).
\end{eqnarray*}
This completes the proof.
\end{proof}
Let $f$, $g$ be elements of $S_k(\varGamma_0(N),\chi)$. The following gives a complex inner product on $S_k(\varGamma_0(N),\chi)$ and is referred to as the Petersson inner product:
$$
\langle f,g\rangle=v(\varGamma_0(N)\backslash\mathfrak H)^{-1}\int_{\varGamma_0(N)\backslash\mathfrak H}f(z)\overline{g(z)}({\rm Im}(z))^kdv(z).
$$
\begin{rem}
On Lemma \ref{rs}, suppose that $s+1-k=0$, $k-l>2$ and $\chi$ is trivial. Then, since $gE_{k-l,N,\overline{\psi}}\in S_k(\varGamma_0(N))$, we have
\begin{eqnarray*}
(4\pi)^{-(k-1)}\Gamma(k-1)D(k-1,f,g)&=&\int_{\varGamma_0(N)\backslash\mathfrak{H}}\overline{f^{\rho}(z)}g(z)E_{k-l,N,\overline{\psi}}(z)({\rm Im}(z))^{k}dv(z)\\
&=&v(\varGamma_0(N)\backslash\mathfrak H)\langle gE_{k-l,N,\overline{\psi}},f^{\rho}\rangle.
\end{eqnarray*}
Therefore $D(k-1,f,g)$ is meaningful.
\end{rem}
Let $t$ be a positive integer. We deal with two linear operators as follows:
\begin{eqnarray*}
S_k(\varGamma_0(N),\chi)\ni f&\longmapsto& f\mid_kB_t\in S_k(\varGamma_0(Nt),\chi)\ \ {\rm for}\ B_t=\begin{pmatrix}t&0\\0&1\end{pmatrix},\\
S_k(\varGamma_0(N),\chi)\ni f&\longmapsto& f\mid_k\omega_N\in S_k(\varGamma_0(N),\overline\chi)\ \ {\rm for}\ \omega_N=\begin{pmatrix}0&-1\\N&0\end{pmatrix}.
\end{eqnarray*}
We define the subspace of $S_k(\varGamma_0(N),\chi)$ consisting of oldforms by
\[
S_k^{\rm old}(\varGamma_0(N),\chi)=\sum_{\scriptstyle c_{\chi} \mid M\mid N \atop {\scriptstyle M \ne N}}\sum_{\hspace{1mm}1\le t \, \mid \,(N/M)}S_k(\varGamma_0(M),\chi)\mid_{k}B_t, 
\]
where $c_{\chi}$ is the conductor of $\chi$, and then 
the subspace of newforms is defined as 
its orthogonal complement with respect to the Petersson inner product: 
\[
S_k^{\rm new}(\varGamma_0(N),\chi)=\left\{f\in S_k(\varGamma_0(N),\chi)\ \left|\ \langle f,g\rangle=0\ \ {\rm for\ any}\ g\in S_k^{\rm old}(\varGamma_0(N),\chi)\right.\right\}.
\]
\begin{defn}
Let $M$ a positive divisor of $N$ and $\chi$ a Dirichlet character modulo $M$. The following linear operator descending the levels of modular forms is called the {\it trace operator}:
$$
\mathrm{Tr}^N_M:S_k(\varGamma_0(N),\chi)\longrightarrow S_k(\varGamma_0(M),\chi)\ \ {\rm as}\ f\longmapsto\hspace{-5mm} \sum_{\gamma\in\varGamma_0(N)\backslash\varGamma_0(M)}\hspace{-5mm}\overline{\chi(\gamma)}f\mid_{k}\gamma.
$$
\end{defn}
For any positive integer $L$ with $M\mid L\mid N$, it is easy to see that
$$
{\rm Tr}^N_M={\rm Tr}^L_M\circ{\rm Tr}^N_L,
$$
that is, the trace operator is independent of the choice of intermediate level. The following two Lemmas are elementary, especially the latter gives the characterization of newforms in terms of the trace operator.
\begin{lem}[cf. {\cite[Lemma 6]{L}}]\label{L.6}
Suppose that $f\in S_k(\varGamma_0(N),\chi)$ and $q$ is a prime dividing $N/c_\chi$. For any positive integer $d$ with $\gcd(d,q)=1$,
$$
{\rm Tr}^{Nd}_{Nd/q}(f\mid_{k}B_d)={\rm Tr}^N_{N/q}(f)\mid_{k}B_d.
$$
\end{lem}
\begin{lem}[cf. {\cite[Theorem 4]{L}}]\label{L.4}
Suppose that $f\in S_k(\varGamma_0(N),\chi)$. Then $f$ belongs to $S_k^{\rm new}(\varGamma_0(N),\chi)$ if and only if for any prime $q$ dividing $N/c_\chi$,
$$
{\rm Tr}^N_{N/q}(f)=0={\rm Tr}^N_{N/q}(f\mid_{k} \omega_N).
$$
\end{lem}
We call $f(z)=\sum_{n=1}^{\infty}c_n(f)\exp(2\pi{\bf i}nz) \in S_k^{\rm new}(\varGamma_0(N),\chi)$ is {\it a primitive form} or {\it primitive at level $N$} 
if it is normalized as $c_1(f)=1$ and is a Hecke eigenform, namely $f$ is a common eigenfunction with respect to all Hecke operators $T(n)$ for any positive integer $n$. As is well-known, $S_k(\varGamma_0(N),\chi)$ has a basis consisting of {\it Hecke eigenforms outside $N$} (i.e. those are Hecke eigenforms for any positive integer $n$ with gcd($n$,$N$)=$1$), and moreover $S_k^{\rm new}(\varGamma_0(N),\chi)$ has a unique basis consisting of primitive forms. 

Meanwhile, {\it if two nonzero Hecke eigenforms outside $N$ of $S_k^{\rm new}(\varGamma_0(N),\chi)$ have the same eigenvalues, they are equal without a constant multiple.}
This statement is said to the {\it multiplicity one property of newforms}. More elaborately, the following assertion holds:
\begin{prop}[Strong Multiplicity One Theorem]\label{M.1}
We define a subset of $S_k(\varGamma_0(N),\chi)$ as follows:
\[
\mathscr{S}_k(\varGamma_0(N),\chi)=\left\{f\in S_k(\varGamma_0(N),\chi)\mid f\textrm{ is a Hecke eigenform outside }N \right\}.
\]
For any $f$, $g\in \mathscr{S}_k(\varGamma_0(N),\chi)$, we say that $f$ and $g$ are equivalent if they have the same eigenvalues outside $N$. Then for each $f\in\mathscr{S}_k(\varGamma_0(N),\chi)$,
\begin{enumerate}
\item[]there exist a unique positive integer $M$ {\rm(}i.e. the conductor of $f${\rm)} and primitive form $f^\circ\in S_k^{\rm new}(\varGamma_0(M),\chi)$ as $c_\chi \mid M\mid N$ and $f^\circ$ is equivalent to $f$, and moreover $f$ is uniquely expressed as follows:
\begin{eqnarray*}
f=\sum_{1\le t\mid (N/M)}c_tf^\circ\mid_{k} B_t\ \ \textrm{for some}\ c_t\in\mathbf C.
\end{eqnarray*}
\end{enumerate}
\end{prop}
\begin{proof}
See {\cite[Lemmas 4.6.2, 4.6.9 and Theorems 4.6.13, 4.6.19]{M}}.
\end{proof}
\begin{rem}
On the above Proposition, especially if $\chi$ is a primitive character modulo $N$ (i.e. $c_\chi=N$) or $f\in S_k^{\rm new}(\varGamma_0(N),\chi)$, then $M=N$ and $f$ is a constant multiple of $f^\circ$.
\end{rem}
By virtue of Proposition \ref{M.1}, we define the following quantity
$$
\gcd(c_{f_1},\dots,c_{f_d}),
$$
where $\{f_1,\dots,f_d\}$ is a basis of $S_k(\varGamma_0(N),\chi)$ consisting of Hecke eigenforms outside $N$ and $c_{f_1},\dots,c_{f_d}$ are those conductors. Then this quantity is determined by $S_k(\varGamma_0(N),\chi)$ only, regardless of the choice of such basis. In fact, let $\{g_1,\dots,g_d\}$ be an other basis consisting of Hecke eigenforms outside $N$. For any $f_i$ with $1\le i\le d$, it immediately follows that there exist some $g_j$ with $1\le j\le d$ as $g_j$ is equivalent to $f_i$. From Proposition \ref{M.1}, this implies that $c_{f_i}=c_{g_j}$. Therefore $\left\{c_{f_1},\dots,c_{f_d}\right\}\subset\left\{c_{g_1},\dots,c_{g_d}\right\}$. Since the conditions of $\{f_1,\dots,f_d\}$ and $\{g_1,\dots,g_d\}$ are fair, the opposite inclusion similarly holds. We call $\gcd(c_{f_1},\dots,c_{f_d})$ {\it the conductor of $S_k(\varGamma_0(N),\chi)$} by abuse of language.

\section{Main theorem}
\begin{thm}\label{gdhm}
Let $k$, $l$, $\lambda$ and $\mu$ be positive integers with $k=l+\lambda$ and $\lambda>2$. Suppose that the conductor of $S_{k\mu}(\varGamma_0(N))$ is $C$. For any $g\in S_l(\varGamma_0(N),\chi)$ and positive integer $M$ with $M\mid C$ and $\gcd(M,N/C)=1$, we have
\begin{equation}
\mathrm{Tr}^N_M((gE_{\lambda,N,\overline{\chi}})^\mu)=c_M\sum_{i=1}^{d}\frac{D(k\mu-1,f_i,g^\mu E^{\mu-1}_{\lambda,N,\overline{\chi}})}{\pi^{k\mu}\langle f_i,f_i\rangle}f_i,\label{0}
\end{equation}
where $c_M=3\cdot4^{-(k\mu-1)}(k\mu-2)![{\rm SL}_2(\mathbf Z):\varGamma_0(M)]^{-1}$ and $\{f_1,\dots,f_d\}$ is a unique basis of $S_{k\mu}^{\rm new}(\varGamma_0(M))$ consisting of primitive forms.
\end{thm}
\begin{proof}
We first show $\mathrm{Tr}^N_M((gE_{\lambda,N,\overline{\chi}})^\mu)\in S_{k\mu}^{\rm new}(\varGamma_0(M))$. For simplicity of description,
put $f=(gE_{\lambda,N,\overline{\chi}})^{\mu}\in S_{k\mu}(\varGamma_0(N))$. Then $f$ is expressed as $f=\sum_{i=1}^rc_ig_i$ for some $c_i\in\mathbf C$ by a basis $\{g_1,\dots,g_r\}$ of $S_{k\mu}(\varGamma_0(N))$ consisting of Hecke eigenforms outside $N$. For each $g_i$ with $1\le i\le r$, it follows from Proposition \ref{M.1} that there exist the conductor $C_i$ of $g_i$ and the equivalent primitive form $g_i^\circ\in S_{k\mu}^{\rm new}(\varGamma_0(C_i))$ such that $g_i=\sum_{1\le t\mid N/C_i}c'_t(g_i^\circ\mid_{k\mu} B_t)$ for some $c'_t\in\mathbf C$, and by the supposition $C=\gcd(C_1,\dots,C_r)$. Therefore, for any prime divisor $q$ of $M$,
\begin{eqnarray*}
{\rm Tr}^M_{M/q}\circ{\rm Tr}^N_M(g_i)\hspace{-2mm}&=&\hspace{-2mm}\sum_{1\le t\mid N/C_i}{\rm }c'_t{\rm Tr}^M_{M/q}\circ{\rm Tr}^N_M(g_i^\circ\mid_{k\mu}B_t)\\
&=&\hspace{-2mm}\sum_{1\le t\mid N/C_i}{\rm }c'_t[\varGamma_0(tC_i):\varGamma_0(N)]{\rm Tr}^M_{M/q}\circ{\rm Tr}^{tC_i}_M(g_i^\circ\mid_{k\mu}B_t)\\
&=&\hspace{-2mm}\sum_{1\le t\mid N/C_i}{\rm }c'_t[\varGamma_0(tC_i):\varGamma_0(N)]{\rm Tr}^{tC_i/q}_{M/q}\circ{\rm Tr}^{tC_i}_{tC_i/q}(g_i^\circ\mid_{k\mu}B_t)=0.\\
\end{eqnarray*}
Here the very last equality follows from Lemma \ref{L.4}, and Lemma \ref{L.6} by $\gcd(t,q)=1$ since $\gcd(M,N/C_i)$ is a divisor of $\gcd(M,N/C)=1$, namely
$$
{\rm Tr}^{tC_i}_{tC_i/q}(g_i^\circ\mid_{k\mu}B_t)=({\rm Tr}^{C_i}_{C_i/q}(g_i^\circ))\mid_{k\mu}B_t=0.
$$
Hence 
$$
{\rm Tr}^M_{M/q}\circ{\rm Tr}^N_M(f)=\sum_{i=1}^rc_i{\rm Tr}^M_{M/q}\circ{\rm Tr}^N_M(g_i)=0.
$$
Similarly the following calculation holds:
\begin{eqnarray*}
\lefteqn{{\rm Tr}^M_{M/q}\left({\rm Tr}^N_M(f)\mid_{k\mu}\omega_M\right)}\\
&=&\sum_{i=1}^rc_i\sum_{1\le t\mid N/C_i}c'_t{\rm Tr}^M_{M/q}\left\{{\rm Tr}^N_M(g_i^\circ\mid_{k\mu}B_t)\mid_{k\mu}\omega_M\right\}\\
&=&\sum_{i=1}^rc_i\sum_{1\le t\mid N/C_i}c'_t[\varGamma_0(tC_i):\varGamma_0(N)]{\rm Tr}^M_{M/q}\left\{{\rm Tr}^{tC_i}_M(g_i^\circ\mid_{k\mu}B_t)\mid_{k\mu}\omega_M\right\}\\
&=&\sum_{i=1}^rc_i\sum_{1\le t\mid N/C_i}c'_t[\varGamma_0(tC_i):\varGamma_0(N)]{\rm Tr}^{tC_i/q}_{M/q}\left\{{\rm Tr}^{tC_i}_{tC_i/q}(g_i^\circ\mid_{k\mu}B_t)\mid_{k\mu}\omega_M\right\}\\
&=&\sum_{i=1}^rc_i\sum_{1\le t\mid N/C_i}c'_t[\varGamma_0(tC_i):\varGamma_0(N)]{\rm Tr}^{tC_i/q}_{M/q}\left\{({\rm Tr}^{C_i}_{C_i/q}(g_i^\circ)\mid_{k\mu}B_t)\mid_{k\mu}\omega_M\right\}=0.
\end{eqnarray*}
Again by Lemma \ref{L.4}, we obtain $\mathrm{Tr}^N_M(f)\in S_{k\mu}^{\rm new}(\varGamma_0(M))$, that is,
\begin{equation}
\mathrm{Tr}^N_M\left((gE_{\lambda,N,\overline{\chi}})^\mu\right)=\sum_{i=1}^{d}c''_if_i,\ \ \textrm{for some}\ c''_i\in\mathbf C.\label{2}
\end{equation}
Secondly we express the above $c''_i$ in terms of the Petersson inner product. For any positive integer $n$ with $\gcd(n,N)=1$, the Hecke operator $T(n)$ acting on $S_{k\mu}(\varGamma_0(M))$ is a self-adjoint operator with respect to the inner product. Therefore, for any $j$ with $1\le j\le d$ we have
\begin{equation*}
c_n(f_i)\langle f_i,f_j\rangle=\langle f_i\mid_{k\mu}T(n),f_j\rangle=\langle f_i,f_j\mid_{k\mu}T(n)\rangle=\overline{c_n(f_j)}\langle f_i,f_j\rangle.
\end{equation*}
Here if $i=j$, then $c_n(f_j)=\overline{c_n(f_j)}$. As mentioned above, the $n$-th Fourier coefficient of $f_j^\rho=\overline{f_j(-\overline z)}$ is equal to $\overline{c_n(f_j)}$. Since $f_j$ is primitive at level $M$, it is obvious that $f_j^\rho$ belongs to $S_{k\mu}^{\rm new}(\varGamma_0(M))$ and $c_1(f_j^\rho)=1$. Hence we have $f_j=f_j^\rho$ by Proposition \ref{M.1}. 
If $i\neq j$, we conclude that $\langle f_i,f_j\rangle=0$ since $f_i\ne f_j$. As a result, by multiplying \eqref{2} by $f_j$, 
\begin{equation}
c''_i=\frac{\langle\mathrm{Tr}^N_M((gE_{\lambda,N,\overline{\chi}})^\mu),f_i\rangle}{\langle f_i,f_i\rangle}.\label{5}
\end{equation}
Finally we express $\langle\mathrm{Tr}^N_M((gE_{\lambda,N,\overline{\chi}})^\mu),f_i\rangle$ in terms of the Rankin-Selberg zeta function. Since $dv(z)$ is ${\rm GL}_2^+(\mathbf R)$-invariant, we have
\begin{eqnarray}
\lefteqn{\langle\mathrm{Tr}^N_M((gE_{\lambda,N,\overline{\chi}})^\mu),f_i\rangle}\nonumber\\
&=&v(\varGamma_0(M)\backslash\mathfrak H)^{-1}\int_{\varGamma_0(M)\backslash\mathfrak H}\mathrm{Tr}^N_M((gE_{\lambda,N,\overline{\chi}}(z))^\mu)\overline{f_i(z)}({\rm Im}(z))^{k\mu}dv(z)\nonumber\\
&=&v(\varGamma_0(M)\backslash\mathfrak H)^{-1}\hspace{-3mm}\sum_{\gamma\in\varGamma_0(N)\backslash\varGamma_0(M)}
\int_{\varGamma_0(M)\backslash\mathfrak H}((gE_{\lambda,N,\overline{\chi}}(z))^\mu\mid_{k\mu}\gamma)\overline{f_i(z)}({\rm Im}(z))^{k\mu}dv(z)\nonumber\\
&=&v(\varGamma_0(M)\backslash\mathfrak H)^{-1}\hspace{-3mm}\sum_{\gamma\in\varGamma_0(N)\backslash\varGamma_0(M)}
\int_{\varGamma_0(M)\backslash\mathfrak H}(gE_{\lambda,N,\overline{\chi}}(\gamma z))^\mu\overline{f_i(\gamma z)}({\rm Im}(\gamma\,z))^{k\mu}dv(\gamma z)\nonumber\\
&=&v(\varGamma_0(M)\backslash\mathfrak H)^{-1}\int_{\varGamma_0(N)\backslash\mathfrak H}(gE_{\lambda,N,\overline{\chi}}(z))^\mu\overline{f_i(z)}({\rm Im}(z))^{k\mu}dv(z).\nonumber
\end{eqnarray}
Here the very last equality is due to the disjoint partition
$$
\varGamma_0(N)\backslash\mathfrak H=\bigsqcup_{\gamma\in \varGamma_0(N)\backslash\varGamma_0(M)}\gamma\cdot(\varGamma_0(M)\backslash\mathfrak H),
$$
and replacing $\gamma\,z$ by $z$. On the other hand, by Lemma \ref{rs} for $s=k\mu-1$, $f_i\in S^{\rm new}_{k\mu}(\varGamma_0(M))\subset S_{k\mu}(\varGamma_0(N))$ and $g^\mu E_{\lambda,N,\overline{\chi}}^{\mu-1}\in S_{k\mu-\lambda}(\varGamma_0(N),\chi)$, we have
$$
(4\pi)^{-(k\mu-1)}\Gamma(k\mu-1)D(k\mu-1,f_i,g^\mu E_{\lambda,N,\overline{\chi}}^{\mu-1})=\int_{\varGamma_0(N)\backslash\mathfrak H}
\hspace{-7mm}
\overline{f_i(z)}(gE_{\lambda,N,\overline{\chi}}(z))^\mu({\rm Im}(z))^{k\mu}dv(z),
$$
where note that $f_i=f_i^\rho$. Consequently, it immediately follows that
\begin{eqnarray}
\lefteqn{\langle\mathrm{Tr}^N_M((gE_{\lambda,N,\overline{\chi}})^\mu),f_i\rangle}\label{4}\\
&=&v(\varGamma_0(M)\backslash\mathfrak H)^{-1}(4\pi)^{-(k\mu-1)}\Gamma(k\mu-1)D(k\mu-1,f_i,g^\mu E_{\lambda,N,\overline{\chi}}^{\mu-1}).\nonumber
\end{eqnarray}
Combining \eqref{4}, \eqref{5}, \eqref{2} and $v(\varGamma_0(M)\backslash\mathfrak H)=(\pi/3)[{\rm SL}_2(\mathbf Z):\varGamma_0(M)]$, 
we have the desired equation \eqref{0}, thereby completing the proof of Theorem \ref{gdhm}.
\end{proof}
\begin{rem}\label{3-1}
On the above proof, note that it was shown that $\mathbf Q_{f_i}$ is contained in $\mathbf R$ for any $i$ with $1\le i\le d$.
\end{rem}
\begin{rem}
The weight $k$ is substantially more than $4$ and even. The reason why is that $\lambda$ is more than $2$ and any modular form with the trivial character of odd weight for $\varGamma_0(N)$ is $0$ only.
\end{rem}

\section{Application}
In this section, we show a generalization of the formula \eqref{e0.1} as an application of Theorem \ref{gdhm}, furthermore by adding a certain assumption, derive the formula \eqref{e0.1} from the generalized formula. We first introduce required knowledge from theory of elliptic curves over number fields.

Let $K$ be an algebraic number field not algebraically closed and satisfying $K\cap\mathbf R=\mathbf Q$, and ${\rm E}_K$ an elliptic curve over $K$ defined by the Weierstrass equation 
$$
{\rm E}_K:y^2+a_1xy+a_3y=x^3+a_2x^2+a_4x+a_6,
$$
where the coefficients $a_1,\dots,a_6\in K$. We regard ${\rm E}_K$ as furnished with the point at infinity. By the change of variables
$$
X=x+\frac{a_1^2+4a_2}{12}\ \ {\rm and}\ \ Y=2y+a_1x+a_3,
$$
the corresponding Weierstrass canonical form is
$$
{\rm E}:Y^2=4X^3-g_2X-g_3.
$$
Here, there are relations among the coefficients of ${\rm E}_K$ and ${\rm E}$ as follows:
\begin{eqnarray*}
12g_2&=&(a_1^2+4a_2)^2-24(a_1a_3+2a_4),\\
216g_3&=&-(a_1^2+4a_2)^3+36(a_1^2+4a_2)(a_1a_3+2a_4)-216(a_2^3+4a_6).
\end{eqnarray*}
Note that $g_2$, $g_3\in K$ with ${g}_2^3-27{g}_3^2\neq0$, and these are uniquely determined by ${\rm E}_K$. Hence from the beginning, we may assume that ${\rm E}_K$ is given by a Weierstrass canonical form. Now the $J$-function is defined by
$$
J:{\rm SL}_2(\mathbf Z)\backslash\mathfrak H\longrightarrow\mathbf C\ \ {\rm as}\ \ J(z)=\frac{12^3g_2(z)^3}{g_2(z)^3-27g_3(z)^2},
$$
where 
\begin{eqnarray*}
g_2(z)=60\hspace{-5mm}\sum_{\scriptstyle (m,n)\in \mathbf Z^2 \atop\scriptstyle{(m,n)\neq(0,0)}}\hspace{-5mm}(mz+n)^{-4}\ \ {\rm and}\ \  g_3(z)=140\hspace{-5mm}\sum_{\scriptstyle (m,n)\in \mathbf Z^2 \atop\scriptstyle{(m,n)\neq(0,0)}}\hspace{-5mm}(mz+n)^{-6}.
\end{eqnarray*}
Let $L$ be a lattice of $\mathbf C$, namely $L=\mathbf Z\omega_1+\mathbf Z\omega_2$ with $\omega_1/\omega_2\in\mathfrak H$.
We put
\begin{eqnarray*}
g_2(L)=g_2(\omega_1,\omega_2)=\omega_2^{-4}g_2(\omega_1/\omega_2),\ \ \ \ g_3(L)=g_3(\omega_1,\omega_2)=\omega_2^{-6}g_3(\omega_1/\omega_2).
\end{eqnarray*}
Since the $J$-function is a bijection from ${\rm SL}_2(\mathbf Z)\backslash\mathfrak H$ to $\mathbf C$, for $g_2$, $g_3\in K$ with $g_2^3-27g_3^2\neq0$, there exists a lattice $L$ such that $g_2(L)=g_2$ and $g_3(L)=g_3$. In other words, for an elliptic curve ${\rm E}$ over $K$ defined by a Weierstrass canonical form, there exists a lattice $L$ such that ${\rm E}={\rm E}_L$, where
$$
{\rm E}_L=\left\{(x,y)\in{\mathbf C}^2\mid y^2=4x^3-g_2(L)x-g_3(L)\ \textrm{with}\ g_2(L)^3-27g_3(L)^2\ne0\right\}.
$$
Moreover the Weierstrass $\wp$-function related to $L$
$$
\wp(z;L)=z^{-2}+\sum_{\omega\in L\backslash\{0\}}\left\{(z-\omega)^{-2}-\omega^{-2}\right\}\ \ \textrm{for any}\ z\in\mathbf C,
$$
which induces the analytic group-isomorphism
$$
\mathbf C/L\longrightarrow {\rm E}_L\ \ {\rm by}\ \ z\longmapsto(\wp(z;L),\wp'(z;L)),
$$
where the base point $0$ corresponds to the point at infinity.
As a result, we see that for a given elliptic curve ${\rm E}_K$, there is a lattice $L=\mathbf Z\omega_1+\mathbf Z\omega_2$ as ${\rm E}_K\simeq\mathbf C/L$, and that 
$$
J(\omega_1/\omega_2)=\frac{12^3g_2(L)^3}{g_2(L)^3-27g_3(L)^2}=\frac{12^3{g_2}^3}{{g_2}^3-27{g_3}^2}\in K,
$$
that is, the $J$-invariant of ${\rm E}_K$ belongs to $K$.

Here we define the algebraic numeber field $K_N$, which is a subfield of the field over $K$ of $N$-division points of ${\rm E}_K$, by
$$
K_N=K(J(\omega_1/\omega_2),J(N\omega_1/\omega_2)).
$$
Especially when $N=1$, note that $K_1=K$.

For any positive integer $M$ dividing $N$, put a lattice $L_M$ containing $L$ as $L_M=\mathbf Z\omega_1+\mathbf Z\omega_2/M$. Then $L_M/L$ is a cyclic subgroup of $\mathbf C/L$ with order $M$. We denote by $S_M$ the embedding of $L_M/L$ into ${\rm E}_K$. 

Let us consider the specialization of modular forms for $\varGamma_0(M)$ at a pair $({\rm E}_K,S_M)$; for any $f\in G_k(\varGamma_0(M))$ with a positive integer $k$, the value
$$
(2\pi/\omega_2)^kf(\omega_1/\omega_2)
$$
is dependent on a pair $({\rm E}_K,S_M)$ only, that is, it is independent of the choice of bases $(\omega_1,\omega_2)$ and $(\omega_1,\omega_2/M)$ of $L$ and $L_M$. In fact, let $(\omega'_1,\omega'_2)$ and $(\omega'_1,\omega'_2/M)$ be another such bases. Then it immediately follows that ${}^t\!(\omega'_1,\omega'_2)=\gamma\ {}^t\!(\omega_1,\omega_2)$ for some $\gamma\in\varGamma_0(M)$. Therefore we conclude that
$$
(2\pi/\omega'_2)^kf(\omega'_1/\omega'_2)=(2\pi/\omega_2)^kf(\omega_1/\omega_2).
$$
\begin{rem}
If $({\rm E'}_K,S_M')$ is equivalent to $({\rm E}_K,S_M)$, then it holds that $\omega'_1/\omega'_2=\gamma\omega_1/\omega_2$ for some $\gamma\in\varGamma_0(M)$, where $(\omega'_1,\omega'_2)$ and $(\omega_1,\omega_2)$ are bases of the lattices corresponding ${\rm E'}_K$ and ${\rm E}_K$ respectively. Hence we have
$$
{\omega'_2}^k(2\pi/\omega'_2)^kf(\omega'_1/\omega'_2)=(c\omega_1+d\omega_2)^k(2\pi/\omega_2)^kf(\omega_1/\omega_2),
$$
where $(c,d)$ is the second row of $\gamma$. That is, the value $(2\pi/\omega_2)^kf(\omega_1/\omega_2)$ is not uniquely determined for the equivalence class of a pair $({\rm E}_K,S_M)$.
\end{rem} 
On the other hand, the transformation polynomial for $g\in G_{k}(\varGamma_0(N))$ to level $M$ is defined by
$$
\Phi^N_M(X;g)=\prod_{\gamma\in\varGamma_0(N)\backslash\varGamma_0(M)}(X-g\mid_k\gamma)=\sum_{i=0}^{\mu(M,N)}(-1)^{i}s_i(g)X^{\mu(M,N)-i},
$$
where $\mu(M,N)=[\varGamma_0(M):\varGamma_0(N)]$ and the coefficient $s_i(g)$ is the elementary symmetric polynomial of degree $i$ with respect to $\left\{g\mid_k\gamma\ ; \gamma\in\varGamma_0(N)\backslash\varGamma_0(M)\right\}$. It is easy to see that $s_i(g)$ belongs to $G_{ki}(\varGamma_0(M))$. The equation $\Phi^N_M(X;g)=0$ is so called the {\it transformation equation} for $g$ to level $M$. This notion is also defined for meromorphic modular forms. In particular, put $J_N(z)=J(Nz)$, then $\Phi^N_1(X;J_N(z))=0$ is classically called the {\it modular equation} of level $N$. By the reason mentioned above, we define the specialization of modular forms for $\varGamma_0(M)$ at $({\rm E}_K,S_M)$ as follows:
\begin{defn}
Let $k$, $M$ be positive integers with $M\mid N$. Suppose that ${\rm E}_K\simeq\mathbf C/L$ and $S_M\simeq L_M/L$ with $L=\mathbf Z\omega_1+\mathbf Z\omega_2$ and $L_M=\mathbf Z\omega_1+\mathbf Z\omega_2/M$. For any $f\in G_k(\varGamma_0(M))$ and $g\in G_k(\varGamma_0(N))$,
\begin{eqnarray*}
f(({\rm E}_K,S_M))&=&(2\pi/\omega_2)^kf(\omega_1/\omega_2),\\
\Phi^N_M(X;g,({\rm E}_K,S_M))&=&\sum_{i=0}^{\mu(M,N)}(-1)^{i}s_i(g)(({\rm E}_K,S_M))X^{\mu(M,N)-i}.
\end{eqnarray*}
Especially when $M=1$, we simply write $f({\rm E}_K)$ instead of $f(({\rm E}_K,S_1))$ since it depends on ${\rm E}_K$ only.
\end{defn}
Let us take the following modular group
$$
\varGamma_1(N)=\left\{\left.\begin{pmatrix}a&b\\c&d\end{pmatrix}\in\varGamma_0(N)\right|a\equiv d\equiv1\pmod N\right\}.
$$
Then $S_k(\varGamma_1(N))$ and $G_k(\varGamma_1(N))$ are expressed as follows:
$$
S_k(\varGamma_1(N))=\bigoplus_{\chi}S_k(\varGamma_0(N),\chi)\ \ {\rm and}\ \ G_k(\varGamma_1(N))=\bigoplus_{\chi}G_k(\varGamma_0(N),\chi),
$$
where $\chi$ runs over all Dirichlet characters modulo $N$. According to this fact, we see that a basis of $S_k(\varGamma_1(N))$ and $G_k(\varGamma_1(N))$ consists of the bases of all $S_k(\varGamma_0(N),\chi)$ and $G_k(\varGamma_0(N),\chi)$, respectively.
The following Proposition assures the existence of a basis of $S_k(\varGamma_1(N))$ consisting of members with rational Fourier coefficients at $\infty$.
\begin{prop}[cf. {\cite[Theorem 3.52.]{S1}}]\label{s1}
Let $\varGamma$ be a modular group such as $\varGamma_1(N)\subset\varGamma\subset\varGamma_0(N)$.
If $l\ge2$, then $S_{l}(\varGamma)$ has a basis consisting of cusp forms with Fourier coefficients at $\infty$ of rational integers.
\end{prop}
Let $S_k(\varGamma_0(N),\chi;K)$ and $G_k(\varGamma_0(N),\chi;K)$ denote the $K$-linear subspace of $S_k(\varGamma_0(N),\chi)$ and $G_k(\varGamma_0(N),\chi)$ consisting of all elements with $K$-rational Fourier coefficients, respectively. For any even integer $l$ with $l\ge4$, we may take the basis $\{h_1,\dots,h_d\}$ of $G_{l}({\rm SL}_2(\mathbf Z);\mathbf Q)$ as follows:
let $(a,b)$ be a unique pair of non-negative integers satisfing $4a+6b=l-12(d-1)$. For any $j$ with $1\le j\le d$,
\begin{eqnarray}
h_j(z)&=&E_{4,1}(z)^aE_{6,1}(z)^{b+2(d-j)}\left\{(2\pi)^{-12}\Delta(z)\right\}^{j-1}\label{e.4.10}\\
&=&\sum_{i=0}^\infty c_i(h_j)q^i,\ \ q=\exp(2\pi{\bf i}z),\nonumber
\end{eqnarray}
where it is obvious by difinition that $c_i(h_j)\in\mathbf Q$ especially $c_i(h_j)=1$ or $0$ for $i=j-1$ or $i<j-1$, respectively, and $\Delta$ is the discriminant function
$$
\Delta(z)=g_2(z)^3-27g_3(z)^2=(2\pi)^{12}q\prod_{n=1}^{\infty}(1-q^n)^{24}.
$$ 
Then for ${\rm E}_K\simeq\mathbf C/L$ with $L=\mathbf Z\omega_1+\mathbf Z\omega_2$, the following equalities hold:
\begin{eqnarray*}
E_{4,1}(\omega_1/\omega_2)=12(2\pi/\omega_2)^{-4}g_2(L),\ \ E_{6,1}(\omega_1/\omega_2)=216(2\pi/\omega_2)^{-6}g_3(L),
\end{eqnarray*}
$$
(2\pi)^{-12}\Delta(\omega_1/\omega_2)=(2\pi/\omega_2)^{-12}\left\{g_2(L)^3-27g_3(L)^2\right\}.
$$
Therefore we have
$$
(2\pi/\omega_2)^lh_j(\omega_1/\omega_2)=\left\{12g_2(L)\right\}^a\hspace{-1mm}\left\{216g_3(L)\right\}^{b+2(d-j)}\hspace{-1mm}\left\{g_2(L)^3-27g_3(L)^2\right\}^{j-1}\hspace{-2mm}\in K.
$$
\begin{rem}\label{rem-1}
For any $f\in G_{l}(\varGamma_0(N);K)$, the value $(2\pi/\omega_2)^{l}f(\omega_1/\omega_2)$ belongs to $K_N$. In fact, we take any member $h_j$ of the basis \eqref{e.4.10} and put $\varphi_j=f/h_j\in K(J,J_N)$. Then, since $(2\pi/\omega_2)^{l}h_j(\omega_1/\omega_2)\in K$,
$$
(2\pi/\omega_2)^{l}f(\omega_1/\omega_2)=(2\pi/\omega_2)^{l}h_j(\omega_1/\omega_2)\varphi_j(\omega_1/\omega_2)\in K_N.
$$
\end{rem}
Let $P_k(\varGamma_0(N))$ denote the subset of $S_k^{\rm new}(\varGamma_0(N))$ consisting of all primitive forms and $\{f_1,\dots,f_d\}$ be a unique basis of $S_k^{\rm new}(\varGamma_0(N))$ consisting of primitive forms. Then it is easy to see that $P_k(\varGamma_0(N))=\{f_1,\dots,f_d\}$. 
We define ${\rm Aut}(\mathbf C)$-action on $G_k(\varGamma_0(N))$ by $f^\sigma$ which means that $\sigma\in{\rm Aut}(\mathbf C)$ acts on all the Fourier coefficients of $f\in G_k(\varGamma_0(N))$. Then $P_k(\varGamma_0(N))$ is stable under this action. The following result is concerned with the algebraicity of the special values of the Rankin-Selberg zeta function.
\begin{prop}[cf. {\cite[Theorem 3]{S2}}]\label{s}
Let $k$, $l$ be positive integers with $k>l$ and $f$ a primitive form of $S_k(\varGamma_1(N))$, $g$ an element of $G_l(\varGamma_1(N))$. For any integer $m$ with $2^{-1}(k+l-2)<m<k$,
\begin{equation*}
the\ value\ \frac{D(m,f,g)}{\pi^k\langle f,f\rangle}\ belongs\ to\ \mathbf Q_f\mathbf Q_g.
\end{equation*}
Moreover, for every $\sigma\in{\rm Aut}(\mathbf C)$, we have
\begin{equation*}
\left\{\frac{D(m,f,g)}{\pi^k\langle f,f\rangle}\right\}^\sigma=\frac{D(m,f^\sigma,g^\sigma)}{\pi^k\langle f^\sigma,f^\sigma\rangle}.
\end{equation*}
\end{prop}
A generalization of the formula \eqref{e0.1} is now stated as follows:
\begin{cor}\label{gdhm-c}
Let $k$, $l$, $\lambda$, $\mu$ and $M$ be taken as Theorem \ref{gdhm}, and furthermore $\{f_{(1)},\dots,f_{(w)}\}$ a complete set of representatives for ${\rm Aut}(\mathbf C)$-orbits of $P_{k\mu}(\varGamma_0(M))$. Suppose that $E_{\lambda,N,\overline{\chi}}$ has all the Fourier coefficients in $K$, $g\in S_l(\varGamma_0(N), \chi ;K)$ and ${\rm E}_K\simeq\mathbf C/L$ with $L=\mathbf Z\omega_1+\mathbf Z\omega_2$ satisfying the following two conditions:
\begin{enumerate}
\item[{\rm (A)}]$\Phi^N_M(X;gE_{\lambda,N,\overline{\chi}},({\rm E}_K,S_M))\in K_M[X]$ is irreducible over $K_M$.
\item[{\rm (B)}]$K_M\cap\mathbf Q_{f(i)}=\mathbf Q$ for any $i$ with $1\le i\le w$.
\end{enumerate}
Then we have
\begin{eqnarray*}
&&{\rm Tr}_{K_N/K_M}\left\{(2\pi/\omega_2)^kgE_{\lambda,N,\overline{\chi}}(\omega_1/\omega_2)\right\}^\mu\\
&=&c_M\sum_{i=1}^{w}{\rm Tr}_{\mathbf Q_{f_{(i)}}K_M/K_M}\left\{\frac{D(k\mu-1,f_{(i)},g^{\mu}E_{\lambda,N,\overline{\chi}}^{\mu-1})}{\pi^{k\mu}\langle f_{(i)},f_{(i)}\rangle}(2\pi/\omega_2)^{k\mu}f_{(i)}(\omega_1/\omega_2)\right\}.
\end{eqnarray*}
\end{cor}
\begin{proof}
Put $P_{k\mu}(\varGamma_0(M))=\{f_1,\dots,f_d\}$. From Theorem \ref{gdhm}, we have
\begin{eqnarray}
\mathrm{Tr}^N_M((gE_{\lambda,N,\overline{\chi}})^\mu)&=&c_M\sum_{n=1}^{d}\frac{D(k\mu-1,f_n,g^\mu E_{\lambda,N,\overline{\chi}}^{\mu-1})}{\pi^{k\mu}\langle f_n,f_n\rangle}f_n\label{4-2}\\
&=&c_M\sum_{i=1}^{w}\sum_{j=1}^{d/w}\frac{D(k\mu-1,{f_{(i)}}^{\sigma_{i,j}},g^\mu E_{\lambda,N,\overline{\chi}}^{\mu-1})}{\pi^{k\mu}\langle {f_{(i)}}^{\sigma_{i,j}},{f_{(i)}}^{\sigma_{i,j}}\rangle}{f_{(i)}}^{\sigma_{i,j}},\nonumber
\end{eqnarray}
where $f_n={f_{(i(n))}}^{\sigma_{i(n),j(n)}}$ for some $\sigma_{i,j}\in{\rm Aut}(\mathbf C/\mathbf Q_{f(i)})\backslash{\rm Aut}(\mathbf C)$, namely 
$$
\{f_1,\dots,f_d\}=\bigsqcup_{i=1}^{w}\left\{{f_{(i)}}^{\sigma_{i,1}},\dots,{f_{(i)}}^{\sigma_{i,d/w}}\right\},
$$
where let $\sigma_{i,1}$ be the identity map. Then, since 
$$
{\rm Aut}(\mathbf C)={\rm Aut}(\mathbf C/\mathbf Q_{f(i)}\cap K_M)={\rm Aut}(\mathbf C/\mathbf Q_{f(i)}){\rm Aut}(\mathbf C/K_M)\ \ {\rm by\ (B)},
$$
we may choose all representatives of ${\rm Aut}(\mathbf C/\mathbf Q_{f(i)})\backslash{\rm Aut}(\mathbf C)$ from ${\rm Aut}(\mathbf C/K_M)$. By Proposition \ref{s1}, there is a basis $\{p_1,\dots,p_d\}$ of $S_{k\mu}^{\rm new}(\varGamma_0(M))$ consisting of members with rational Fourier coefficients. Then ${f_{(i)}}^{\sigma_{i,j}}=\sum_{s=1}^\infty c_{s}({f_{(i)}}^{\sigma_{i,j}})q^{s}$ is expressed by this basis:
$$
{f_{(i)}}^{\sigma_{i,j}}=\sum_{s=1}^\infty c_{s}({f_{(i)}}^{\sigma_{i,j}})q^{s}=\sum_{{t}=1}^{d}a_tp_{t}=\sum_{t=1}^{d}\sum_{s=1}^\infty a_tc_{s}(p_{t})q^{s}\ \ \text{for some}\ a_{t}\in\mathbf C,
$$
where $c_{s}(p_{t})\in\mathbf Q$ for all $s$, $t$. By comparing the coefficient of $q^{s}$, we have $c_{s}({f_{(i)}}^{\sigma_{i,j}})=\sum_{{t}=1}^{d}a_tc_{s}(p_{t})$, and by arranging $c_{1}({f_{(i)}}^{\sigma_{i,j}}),\dots,c_{d}({f_{(i)}}^{\sigma_{i,j}})$ to the column vector, the following linear equality is obtained:
$$
\begin{pmatrix}c_1({f_{(i)}}^{\sigma_{i,j}})\\\vdots\\c_d({f_{(i)}}^{\sigma_{i,j}})\end{pmatrix}=
\begin{pmatrix}c_1(p_1)&\dots&c_1(p_d)\\\vdots&\ddots&\vdots\\c_d(p_1)&\dots&c_d(p_d)\end{pmatrix}
\begin{pmatrix}a_1\\\vdots\\a_d\end{pmatrix}.
$$
Since $p_1$,\dots,$p_d$ are $\mathbf C$-linear independent, the above matrix $(c_s(p_t))$ belongs to ${\rm GL}_d(\mathbf Q)$.
Therefore we have $a_t\in\mathbf Q_{{f_{(i)}}^{\sigma_{i,j}}}$ for any $t$ with $1\le t\le d$, and by Remark \ref{rem-1} we conclude that 
$$
{f_{(i)}}^{\sigma_{i,j}}(({\rm E}_K,S_M))=\sum_{t=1}^d{a_t}p_t(({\rm E}_K,S_M))\in\mathbf Q_{{f_{(i)}}^{\sigma_{i,j}}}K_M.
$$
Moreover since $a_t$ is a $\mathbf Q$-linear combination with respect to $\left\{c_s({f_{(i)}}^{\sigma_{i,j}})\right\}_{s=1}^d$, for any $\tau\in{\rm Aut}(\mathbf C/K_M)$,
\begin{eqnarray*}
{f_{(i)}}^{\sigma_{i,j}}(({\rm E}_K,S_M))^\tau=\sum_{t=1}^d{a_t}^\tau p_t(({\rm E}_K,S_M))={f_{(i)}}^{{\sigma_{i,j}}\tau}(({\rm E}_K,S_M)).
\end{eqnarray*}
Then we specialize the equation \eqref{4-2} at $({\rm E}_K,S_M)$:
\begin{eqnarray*}
\lefteqn{\mathrm{Tr}^N_M((gE_{\lambda,N,\overline{\chi}})^\mu)({\rm E}_K,S_M)}\\
&=&c_M\sum_{i=1}^{w}\sum_{j=1}^{d/w}\frac{D(k\mu-1,{f_{(i)}}^{\sigma_{i,j}},g^\mu E_{\lambda,N,\overline{\chi}}^{\mu-1})}{\pi^{k\mu}\langle {f_{(i)}}^{\sigma_{i,j}},{f_{(i)}}^{\sigma_{i,j}}\rangle}{f_{(i)}}^{\sigma_{i,j}}(({\rm E}_K,S_M))
=c_M\sum_{i=1}^{w}\sum_{j=1}^{d/w}\xi_{{f_{(i)}}^{\sigma_{i,j}}},\\
&&\textrm{where\ put}\ \ \xi_{{f_{(i)}}^{\sigma_{i,j}}}=\frac{D(k\mu-1,{f_{(i)}}^{\sigma_{i,j}},g^\mu E_{\lambda,N,\overline{\chi}}^{\mu-1})}{\pi^{k\mu}\langle {f_{(i)}}^{\sigma_{i,j}},{f_{(i)}}^{\sigma_{i,j}}\rangle}{f_{(i)}}^{\sigma_{i,j}}(({\rm E}_K,S_M)).
\end{eqnarray*}
Since $g^\mu E_{\lambda,N,\overline{\chi}}^{\mu-1}$ has all the Fourier coefficients in $K$, by Proposition \ref{s} we have $\xi_{{f_{(i)}}^{\sigma_{i,j}}}\in\mathbf Q_{{f_{(i)}}^{\sigma_{i,j}}}K_M$ and for any $\tau\in{\rm Aut}(\mathbf C/K_M)$,
\begin{equation*}
{\xi_{{f_{(i)}}^{\sigma_{i,j}}}}^\tau=\xi_{{f_{(i)}}^{\sigma_{i,j}\tau}}\in\{\xi_{{f_{(i)}}^{\sigma_{i,j}}}\}_{j=1}^{d/w}.
\end{equation*}
Therefore we eventually have
\begin{equation*}
\{\xi_{{f_{(i)}}^\tau}\ ;\ \tau\in{{\rm Aut}(\mathbf C/K_M)}\}=\{\xi_{{f_{(i)}}^{\sigma_{i,j}}}\}_{j=1}^{d/w},\ \ {\rm that\ is,}
\end{equation*}
$\sum_{j=1}^{d/w}\xi_{{f_{(i)}}^{\sigma_{i,j}}}=\mathrm{Tr}_{\mathbf Q_{f_{(i)}}K_M/K_M}(\xi_{f_{(i)}})$. 
On the other hand, by definition,
$$
\mathrm{Tr}^N_M((gE_{\lambda,N,\overline{\chi}})^\mu)({\rm E}_K,S_M)=\sum_{\gamma\in\varGamma_0(N)\backslash\varGamma_0(M)}\left\{\left.\left(2\pi/\omega_2\right)^kgE_{\lambda,N,\overline{\chi}}\right\}^\mu\right|_{k\mu}\gamma\left(\omega_1/\omega_2\right).
$$
Here $\left\{(2\pi/\omega_2)^kgE_{\lambda,N,\overline{\chi}}(\omega_1/\omega_2)\right\}^\mu$ belongs to $K_N$ by Remark \ref{rem-1}, and moreover by (A), all $K_M$-isomorphisms of $K_N$ into $\mathbf C$ are given by all elements of $\varGamma_0(N)\backslash\varGamma_0(M)$. Hence we have
$$
\hspace{-3mm}\sum_{\gamma\in\varGamma_0(N)\backslash\varGamma_0(M)}\hspace{-7mm}\left\{\left.\left(2\pi/\omega_2\right)^kgE_{\lambda,N,\overline{\chi}}\right\}^\mu\right|_{k\mu}\hspace{-2mm}\gamma\left(\omega_1/\omega_2\right)=\mathrm{Tr}_{K_N/K_M}\hspace{-1mm}\left\{(2\pi/\omega_2)^kgE_{\lambda,N,\overline{\chi}}(\omega_1/\omega_2)\right\}^\mu.
$$
This proves our Corollary.
\end{proof}
\begin{rem}
On the above proof, it immediately follows that $\{\xi_{{f_{(i)}}^{\sigma_{i,j}}}\}_{j=1}^{d/w}$ is stable under ${\rm Aut}(\mathbf C/K_M)$.
Therefore $F(X)=\prod_{j=1}^{d/w}(X-\xi_{{f_{(i)}}^{\sigma_{i,j}}})$ belongs to $K_M[X]$. Moreover, since $\xi_{{f_{(i)}}^{\sigma_{i,j}}}\neq\xi_{{f_{(i)}}^{\sigma_{i,{j'}}}}$ for any $j\neq j'$, we conclude that $F(X)$ is irreducible over $K_M$, namely $\mathbf Q_{f_{(i)}}K_M=K_M(\xi_{f_{(i)}})$ and the dimension of $\mathbf Q_{f_{(i)}}K_M$ over $K_M$ is $d/w$.
\end{rem}
\begin{rem}\label{rem-2}
If the algebraic number field $K$ is algebraically closed, then (A) does not hold, namely $\Phi^N_M(X;gE_{\lambda,N,\overline{\chi}},({\rm E}_K,S_M))\in K_M[X]$ is reducible over $K_M$ for any $g$, ${\rm E}_K$. Furthermore if $K\cap\mathbf R\neq\mathbf Q$, then (B) generally does not hold by Remark \ref{3-1}. For these reasons, we supposed that $K$ is not algebraically closed and satisfies $K\cap\mathbf R=\mathbf Q$. Meanwhile, we see that (A), (B) hold whenever take $g$, ${\rm E}_K:y^2=4x^3-g_2x-g_3$ such as $g\in S_l(\varGamma_0(M),\chi;K)\subset S_l(\varGamma_0(N),\chi;K)$, $g_2g_3\neq0$ and $J(M\omega_1/\omega_2)\in\mathbf Q_{f(i)}$ for any $i$ with $1\le i\le w$.
\end{rem}
We specialize to $M=1$, $K=\mathbf Q$ and $\chi$ is trivial. Let us assume the following, which is concerned with the number of ${\rm Aut}(\mathbf C)$-orbits of $P_k({\rm SL}_2(\mathbf Z))$:
\begin{conj}[Maeda's conjecture]\label{M}
There is only one ${\rm Aut}(\mathbf C)$-orbit of $P_{k}({\rm SL}_2(\mathbf Z))$ for all integers $k$ with $12\le k$.
\end{conj}
At least, for not so large $k$, it is verified by calculations that this assertion holds true (cf. e.g. \cite{FJ}). 
Now on Corollary \ref{gdhm-c}, we have the formula \eqref{e0.1} by putting $w=1$ as follows:
\begin{cor}\label{dhm-c}
Let $k$, $l$, $\lambda$ and $\mu$ be taken as Theorem \ref{dhm}. Suppose that $g\in S_l(\varGamma_0(N);\mathbf Q)$ and ${\rm E}_{\mathbf Q}\simeq\mathbf C/L$ with $L=\mathbf Z\omega_1+\mathbf Z\omega_2$ satisfying the following condition:
\begin{enumerate}
\item[{\rm (A)}]$\Phi^N_1(X;gE_{\lambda,N},{\rm E}_{\mathbf Q})\in\mathbf Q[X]$ is irreducible over $\mathbf Q$.
\end{enumerate} 
If Maeda's Conjecture holds true, then we have
\begin{eqnarray*}
\lefteqn{{\rm Tr}_{\mathbf Q_N/\mathbf Q}\left\{(2\pi/\omega_2)^kgE_{\lambda,N}(\omega_1/\omega_2)\right\}^\mu}\\
&=&c_1{\rm Tr}_{\mathbf Q_f/\mathbf Q}\left\{\frac{D(k\mu-1,f,g^{\mu}E_{\lambda,N}^{\mu-1})}{\pi^{k\mu}\langle f,f\rangle}(2\pi/\omega_2)^{k\mu}f(\omega_1/\omega_2)\right\},
\end{eqnarray*}
where $f$ is any element of $P_{k\mu}({\rm SL}_2(\mathbf Z))$.
\end{cor}


\begin{acknow} 
The author would like to thank Hisa-aki Kawamura,
for suggesting this subject and making a comment on early drafts.
\end{acknow}


\end{document}